\documentclass[10pt]{amsart} 
\def\update{May 25, 2013}

\usepackage{amssymb}

\newtheorem{theorem}{Theorem}[]
\newtheorem{lemma}{Lemma}
\newtheorem{corollary}{Corollary}
\newtheorem{proposition}{Proposition}

\newtheorem*{theorem*}{Theorem}
\newtheorem*{proposition*}{Proposition}
\newtheorem*{corollary*}{Corollary}
\newtheorem*{question*}{Question}
\theoremstyle{definition}
\newtheorem{remark}{Remark}
\newtheorem*{remark*}{Remark}

\usepackage[hypertex]{hyperref}

\def\BbbL{{\mathbb L}}

\def\sfK{{\sf K}}
\def\sfS{{\sf S}}

\def\N{{\mathbf N}}
\def\Q{{\mathbf Q}}
\def\R{{\mathbf R}}
\def\Z{{\mathbf Z}}

\def\ue{\underline{e}}

\def\un{\underline{n}}
\def\uq{\underline{q}}
\def\uu{\underline{u}}
\def\kappatilde{\tilde{\kappa}}

\begin{document}

\title[Liouville Numbers, Liouville Sets and Liouville Fields]{ Liouville Numbers, Liouville Sets \\
and Liouville Fields
}

\author{K. Senthil Kumar, R. Thangadurai and M. Waldschmidt} 
\address[K. Senthil Kumar and R. Thangadurai]{Harish-Chandra Research Institute, Chhatnag Road, Jhunsi, Allahbad, 211019, India}
\address[M. Waldschmidt]{Institut de Math\'ematiques de Jussieu, 
Th\'eorie des Nombres Case courrier 247, 
Universit\'e Pierre et Marie Curie (Paris 6), 
PARIS Cedex 05 FRANCE}
\email[K. Senthil Kumar]{senthil@hri.res.in}
\email[R. Thangadurai]{thanga@hri.res.in}
\email[M. Waldschmidt]{miw@math.jussieu.fr}

\begin{abstract} 
Following earlier work by \'{E}.~Maillet 100 years ago, we introduce the definition of a {\it Liouville set}, which extends the definition of a Liouville number. We also define a {\it Liouville field}, which is a field generated by a Liouville set. Any Liouville number belongs to a Liouville set $\sfS$ having the power of continuum and such that $\Q\cup\sfS$ is a Liouville field. 
\\
\null
\hfill {\bf Update: \update }
\end{abstract}
\maketitle

\section{ Introduction}\label{Section:Introduction}

For any integer $q$ and any real number $x \in \R$, we denote by
$$
\Vert qx\Vert = \min_{m\in{ \Z}} |qx -m|
$$
the distance of $qx$ to the nearest integer. Following \'E.~Maillet \cite{Maillet,Maillet1922}, an irrational real number $\xi$ is said to be a {\it Liouville number} if, for each integer $n\geq 1$, there exists an integer $q_n \geq 2$ such that the sequence $\bigl(u_n(\xi)\bigr)_{n\ge 1}$ of real numbers defined by 
$$
u_n(\xi) = -\frac{\log\Vert q_n\xi \Vert}{\log q_n}
$$ 
satisfies 
$\displaystyle \lim_{n\to\infty} u_n(\xi) = \infty$.
If $p_n$ is the integer such that $\Vert q_n\xi\Vert = |\xi q_n - p_n|$, then the definition of $u_n(\xi)$ can be written
$$
\left| q_n \xi - p_n \right| = \frac{1}{q_n^{u_n(\xi)}}\cdotp
$$
An equivalent definition is to saying that a Liouville number is a real number $\xi$ such that, for each integer $n\geq 1$, there exists a rational number $p_n/q_n$ with $q_n \geq 2$ such that
$$
0<\left| \xi-\frac{p_n}{q_n}\right| \le \frac{1}{q_n^n}\cdotp
$$
We denote by ${\BbbL}$ the set of Liouville numbers. Following \cite{liouville1}, any Liouville number is transcendental.

We introduce the notions of a {\it Liouville set} and of a {\it Liouville field}. They extend what was done by \'{E}. Maillet in Chap.~III of \cite{Maillet}.

\smallskip

\noindent{\bf Definition}.
A {\it Liouville set} is a subset $\sfS$ of ${\BbbL}$ for which there exists an increasing sequence $(q_n)_{n\ge 1}$ of positive integers having the following property: for any $\xi\in \sfS$, there exists a sequence $\bigl(b_n\bigr)_{n\ge 1}$ of positive rational integers and there exist two positive constants $\kappa_1$ and $\kappa_2$ such that, for any sufficiently large $n$, 
\begin{equation}\label{Equation:LiouvilleSet}
1\le b_n \leq q_{n}^{\kappa_1} \mbox{ and } \Vert b_n \xi \Vert \le \frac{1}{q_n^{ \kappa_2 n}}\cdotp
\end{equation} 

It would not make a difference if we were requesting these inequalities to hold for any $n\ge 1$: it suffices to change the constants $\kappa_1$ and $\kappa_2$. 

\smallskip

\noindent{\bf Definition}.
A {\it Liouville field} is a field of the form $\Q(\sfS)$ where $\sfS$ is a Liouville set. 

From the definitions, it follows that, for a real number $\xi$, the following conditions are equivalent:
\\
$(i)$ $\xi$ is a Liouville number. 
\\
$(ii)$ $\xi$ 
belongs to some Liouville set.
\\
$(iii)$ The set $\{\xi\}$ is a Liouville set. 
\\
$(iv)$ The field $\Q(\xi)$ is a Liouville field. 

If we agree that the empty set is a Liouville set and that $\Q$ is a Liouville field, then any subset of a Liouville set is a Liouville set, and also (see Theorem $\ref{Theorem:Field}$) any subfield of a Liouville field is a Liouville field. 

\noindent{\bf Definition.} 
Let $\uq=(q_n)_{n\ge 1}$ be an increasing sequence of positive integers and let $\uu=( u_n)_{n\ge 1}$ be a sequence of positive real numbers such that $ u_n \rightarrow\infty $ as $ n\rightarrow\infty$.
Denote by $\sfS_{\uq,\uu}$ the set of $\xi\in{\BbbL}$ such that there exist two positive constants $\kappa_1$ and $\kappa_2$ and there exists a sequence $\bigl(b_n\bigr)_{n\ge 1}$ of positive rational integers with 
$$
1 \leq b_n \leq q_{n}^{\kappa_1} \mbox{ and } 
\Vert b_n \xi \Vert \le 
\frac{1}{q_n^{\kappa_2 u_n}}\cdotp
$$
Denote by $\un$ the sequence $\uu=(u_n)_{n\ge 1}:= (1,2,3,\dots)$ with $u_n=n$ ($n\ge 1$). 
For any increasing sequence $\uq=(q_n)_{n\ge 1}$ of positive integers, we denote by $\sfS_{\uq}$ the 
set 
$\sfS_{\uq,\un}$. 

Hence, by definition, a Liouville set is a subset of some $\sfS_{\uq}$. 
In section $\ref{section:PfLemmaSqu}$ we prove the following lemma: 

\begin{lemma}\label{Lemma:Squ}
For any increasing sequence $\uq$ of positive integers and any sequence $\uu$ of positive real numbers which tends to infinity, the set $\sfS_{\uq,\uu}$
is a Liouville set.
\end{lemma}

Notice that if $(m_n)_{n\ge 1}$ is an increasing sequence of positive integers, then for the subsequence $\uq'=(q_{m_n})_{n\ge 1}$ of the sequence $\uq$, we have $\sfS_{\uq',\uu}\supset \sfS_{\uq,\uu}$.

\medskip

\noindent
{\bf Example}. Let $\uu=(u_n)_{n\ge 1}$ be a sequence of positive real numbers which tends to infinity. Define 
$f:\N\rightarrow \R_{>0}$ by $f(1)=1$ and 
$$
f(n)=u_1u_2\cdots u_{n-1} \qquad (n\ge 2),
$$
so that $f(n+1)/f(n)=u_n$ for $n\ge 1$. Define the sequence $\uq=(q_n)_{n\ge 1}$ by $q_n= \lfloor 2^{f(n)}\rfloor$.
Then, for any real number $t>1$, the number 
$$
\xi_t=\sum_{n\ge 1}
\frac{1}{ \lfloor t^{f(n)}\rfloor}
$$ 
belongs to $\sfS_{\uq,\uu}$. The set $\{\xi_t\; \mid \; t>1\}$ has the power of continuum, since $\xi_{t_1}<\xi_{t_2}$ for $t_1>t_2>1$. 

The sets $\sfS_{\uq,\uu}$ have the following property (compare with Theorem $\mathrm{I}_3$ in \cite{Maillet}): 

\begin{theorem}\label{Theorem:Field}
For any increasing sequence $\uq$ of positive integers and any sequence $\uu$ of positive real numbers which tends to infinity, the set $\Q\cup \sfS_{\uq,\uu}$
is a field.
\end{theorem}

We denote this field by $\sfK_{\uq,\uu}$, and by $\sfK_{\uq}$ for the sequence $\uu=\un$. From Theorem $\ref{Theorem:Field}$, it follows that a field is a Liouville field if and only if it is a subfield of some $\sfK_{\uq}$. Another consequence is that, if $\sfS$ is a Liouville set, then $\Q(\sfS)\setminus\Q$ is a Liouville set. 

It is easily checked that if 
$$
\liminf_{n\rightarrow \infty} \frac{u_n}{u_n'} >0,
$$
then $\sfK_{\uq,\uu}$ is a subfield of $\sfK_{\uq, \uu'}$. In particular if
$$
\liminf_{n\rightarrow \infty} \frac{u_n}{n} >0,
$$
then $\sfK_{\uq,\uu}$ is a subfield of $\sfK_{\uq}$, while if
$$
\limsup_{n\rightarrow \infty} \frac{u_n}{n} <+\infty
$$
then $\sfK_{\uq}$ is a subfield of $\sfK_{\uq,\uu}$.

 If $R\in\Q(X_1,\dots,X_\ell)$ is a rational fraction and if $\xi_1,\dots,\xi_\ell$ are elements of a Liouville set $\sfS$ such that $\eta=R(\xi_1,\ldots,\xi_\ell)$ is defined, then Theorem $\ref{Theorem:Field}$ implies that $\eta$ is either a rational number or a Liouville number, and in the second case $\sfS\cup\{\eta\}$ is a Liouville set. For instance, if, in addition, $R$ is not constant and $\xi_1,\dots,\xi_\ell$ are algebraically independent over $\Q$, then $\eta$ is a Liouville number and $\sfS\cup\{\eta\}$ is a Liouville set. For $\ell=1$, this yields: 

\begin{corollary}\label{Corollary:q(X)}
Let $R\in\Q(X)$ be a rational fraction and let $\xi$ be a Liouville number. Then $R(\xi)$ is a Liouville number and $\{\xi,R(\xi)\}$ is a Liouville set. 
\end{corollary}

We now show that $ \sfS_{\uq,\uu}$ is either empty or else uncountable and we characterize such sets. 

\begin{theorem}\label{Theorem::UncountableLiouvilleSets}
Let $\uq$ be an increasing sequence of positive integers and $\uu = (u_n)_{n\geq 1}$ be an increasing sequence of positive real numbers such that $u_{n+1} \geq u_n +1$. Then 
 the Liouville set $\sfS_{\uq, \uu}$ is non empty if and only if 
$$
\limsup_{n\to\infty} \frac{\log q_{n+1}}{u_n\log q_n} > 0.$$
Moreover, if the set $\sfS_{\uq, \uu}$ is non empty, then it has the power of continuum.

\end{theorem}

Let $t$ be an irrational real number which is not a Liouville number. By a result due to P.~Erd\H{o}s \cite{erd1}, we can write $t=\xi+\eta$ with two Liouville numbers $ \xi$ and $\eta$. Let $\uq$ be an increasing sequence of positive integers and $\uu$ be an increasing sequence of real numbers such that $\xi \in\sfS_{\uq, \uu}$. Since any irrational number in the field $K_{\uq, \uu}$ is in $\sfS_{\uq, \uu}$, it follows that the Liouville number $\eta=t-\xi$ does not belong to $\sfS_{\uq, \uu}$. 

One defines a reflexive and symmetric relation $R$ between two Liouville numbers by $\xi R\eta$ if $\{\xi,\eta\}$ is a Liouville set. The equivalence relation which is induced by $R$ is trivial, as shown by the next result, which is a consequence of Theorem $\ref{Theorem::UncountableLiouvilleSets}$.

\begin{corollary}\label{Corollary:EquivalenceTrivial} 
Let $\xi $ and $\eta$ be Liouville numbers. Then there exists a subset  $\vartheta$ of ${\BbbL}$ having the power of continuum such that, for each such $\varrho \in \vartheta$, both sets $\{\xi,\varrho\}$ and $\{\eta,\varrho\}$ are Liouville sets. 
\end{corollary}

In \cite{Maillet}, \'{E} Maillet introduces the definition of Liouville numbers {\it corresponding} to a given Liouville number. However this definition depends on the choice of a given sequence $\uq$ giving the rational approximations. This is why we start with a sequence $\uq$ instead of starting with a given Liouville number. 

The intersection of two nonempty Liouville sets maybe empty. More generally, we show that there are uncountably many Liouville sets $\sfS_{\uq}$ with pairwise empty intersections. 

\begin{proposition}\label{proposition:UncountablymanyLiouvilleSets}

For $0<\tau<1$, define $\uq^{(\tau)}$ as the sequence $(q^{(\tau)}_n)_{n\ge 1}$ with 
$$
q^{(\tau)}_n=2^{n! \lfloor n^{\tau} \rfloor } \qquad (n\ge 1).
$$ 
Then the sets $\sfS_{\uq^{(\tau)}}$, $0<\tau<1$, are nonempty (hence uncountable) and pairwise disjoint.

\end{proposition}

To prove that a real number is not a Liouville number is most often difficult. But to prove that a given real number does not belong to some Liouville set $\sfS$ is easier. If $\uq'$ is a subsequence of a sequence $\uq$, one may expect that $\sfS_{\uq'}$ may often contain strictly $\sfS_{\uq}$. Here is an example.

\begin{proposition}\label{proposition:SubLiouvilleSets}
Define the sequences $\uq$, $\uq'$ and $\uq''$ by 
$$
q_n=2^{n!}, 
\quad
q'_n=q_{2n}=2^{(2n)!} 
\quad\hbox{and}\quad
q''_n=q_{2n+1}=2^{(2n+1)!} 
\qquad (n\ge 1),
$$
so that $\uq$ is the increasing sequence deduced from the union of $\uq'$ and $\uq''$. 
Let $\lambda_n$ be a sequence of positive integers such that 
$$
\lim_{n\rightarrow \infty} \lambda_n=
\infty
\quad\hbox{and}\quad \lim_{n\rightarrow \infty}\frac{ \lambda_n}{n}=0.
$$
Then the number 
$$
\xi:=\sum_{n\ge 1} \frac{1}{2^{(2n-1)!\lambda_n}}
$$
belongs to 
$\sfS_{\uq'}$ but not to $\sfS_{\uq}$. Moreover
$$
\sfS_{\uq}= \sfS_{\uq'}\cap \sfS_{\uq''}.
$$
\end{proposition}

When $\uq$ is the increasing sequence deduced form the union of $\uq'$ and $\uq''$, we always have $\sfS_{\uq}\subset\sfS_{\uq'}\cap\sfS_{\uq''}$; Proposition \ref{proposition:UncountablymanyLiouvilleSets} gives an example where $\sfS_{\uq'}\ne\emptyset$ and $\sfS_{\uq''}\ne\emptyset$, while $\sfS_{\uq}$ is the empty set. In the example from Proposition \ref {proposition:SubLiouvilleSets}, the set $\sfS_{\uq}$ coincides with $\sfS_{\uq'}\cap \sfS_{\uq''}$. This is not always the case. 

\begin{proposition}\label{proposition:13} There exists two increasing sequences $\uq'$ and $\uq''$ of positive integers with union $\uq$ such that $\sfS_{\uq}$ is a strict nonempty subset of $\sfS_{\uq'}\cap \sfS_{\uq''}$.
\end{proposition}

Also, we prove that given any increasing sequence $\uq$, there exists a subsequence $\uq'$ of $\uq$ such that $\sfS_{\uq}$ is a strict subset of $\sfS_{\uq'}$. More generally, we prove

\begin{proposition}\label{proposition:14}
Let $\uu = (u_n)_{n\geq 1}$ be a sequence of positive real numbers such that for every $n\geq 1$, we have $\sqrt{u_{n+1}} \leq u_n + 1 \leq u_{n+1}$. Then any increasing sequence $\uq$ of positive integers has a subsequence $\uq'$ for which $\sfS_{\uq', \uu}$ strictly contains $\sfS_{\uq, \uu}$. In particular, for any increasing sequence $\uq$ of positive integers has a subsequence $\uq'$ for which $\sfS_{\uq'}$ is strictly contains $\sfS_{\uq}$. 
\end{proposition}

 \begin{proposition}\label{proposition:Squ-densebutnotGdelta} 
 The sets $\sfS_{\uq, \uu}$ are not  $G_\delta$ subsets of ${\Bbb R}$. 
If they are non empty, then they are  dense in ${\Bbb R}$.
\end{proposition}

The proof of lemma $\ref{Lemma:Squ}$ is given in section $\ref{section:PfLemmaSqu}$, 
the proof of Theorem $\ref{Theorem:Field}$ in section $\ref{section:proofField}$, 
the proof of Theorem $\ref{Theorem::UncountableLiouvilleSets}$ in section $\ref{Section:UncountableLiouvilleSetss}$,
the proof of Corollary $\ref{Corollary:EquivalenceTrivial}$ in section $\ref{Section:EquivalenceTrivial}$. 
The proofs of Propositions $\ref{proposition:UncountablymanyLiouvilleSets}$, $\ref{proposition:SubLiouvilleSets}$, $\ref{proposition:13}$ and $\ref{proposition:14}$ are given in section $\ref{section:PfPropositions}$ and the proof of Proposition $\ref {proposition:Squ-densebutnotGdelta}$ is given in section 
$\ref{section:Squ-densebutnotGdelta}$.

\section{Proof of lemma $\ref{Lemma:Squ}$}\label{section:PfLemmaSqu}

\begin{proof}[Proof of Lemma $\ref{Lemma:Squ}$]
Given $\uq$ and $\uu$, define inductively a sequence of positive integers $(m_n)_{n\ge 1}$ as follows. Let $m_1$ be the least integer $m\ge 1$ such that $u_{m}>1$. Once $m_1,\dots,m_{n-1}$ are known, define $m_n$ as the least integer $m>m_{n-1}$ for which 
$u_m>n$. Consider the subsequence $\uq'$ of $\uq$ defined by $q'_n=q_{m_n}$. Then $\sfS_{\uq,\uu}\subset \sfS_{\uq'}$, hence $\sfS_{\uq,\uu}$ is a Liouville set. 
\end{proof}

\begin{remark}\label{Remark:definitionLiouvilleSet}
In the definition of a Liouville set, if assumption $(\ref{Equation:LiouvilleSet})$ is satisfied for some $\kappa_1$, then it is also satisfied with $\kappa_1$ replaced by any $\kappa'_1>\kappa_1$. Hence there is no loss of generality to assume $\kappa_1>1$. Then, in this definition, one could add to 
$(\ref{Equation:LiouvilleSet})$ the condition $q_n \leq 
b_n$. Indeed, if, for some $n$, we have $b_n<q_n $, then we set 
$$
b'_{n}=\left\lceil \frac{q_n}{b_{n}}\right\rceil b_{n},
$$
so that 
$$
q_n\le b'_{n}\le q_n+b_{n}\le 2 q_n.
$$
Denote by $a_n$ the nearest integer to $b_n\xi$ and set 
$$
a'_{n}=\left\lceil \frac{q_n}{b_n}\right\rceil a_{n}.
$$
Then, for $\kappa'_2<\kappa_2$ and, for sufficiently large $n$, we have
$$ 
\bigl| b'_n \xi - a'_{n} \bigr|=
\left\lceil \frac{q_n}{b_{n}}\right\rceil 
\bigl| b_{n} \xi - a_{n} \bigr| \le \frac{q_n}{q_n^{\kappa_2 n}}\le
\frac{1}{(q_n)^{\kappa'_2 n}}\cdotp
$$
Hence condition $(\ref{Equation:LiouvilleSet})$ can be replaced by
$$
q_{n} \leq 
b_{n} \leq q_{n}^{\kappa_1} \mbox{ and } \Vert b_n\xi \Vert \le \frac{1}{q_n^{ \kappa_2 n}}.
$$
Also, one deduces from Theorem $\ref{Theorem::UncountableLiouvilleSets}$,  that the sequence $\bigl(b_n\bigr)_{n\ge 1}$ is increasing for sufficiently large $n$. Note also that same way we can assume that 
$$
q_{n} \leq 
b_{n} \leq q_{n}^{\kappa_1} \mbox{ and } \Vert b_n\xi \Vert \le \frac{1}{q_n^{ \kappa_2 u_n}}.
$$
\end{remark}

\section{Proof of Theorem $\ref{Theorem:Field}$}\label{section:proofField}

We first prove the following: 

\begin{lemma}\label{Lemma:unsuralpha}
Let $\uq$ be an increasing sequence of positive integers and $\uu = (u_n)_{n\geq 1}$ be an increasing sequence of real numbers. Let $\xi\in\sfS_{\uq, \uu}$. Then $1/\xi\in \sfS_{\uq, \uu}$. 
\end{lemma}

As a consequence, if $\sfS$ is a Liouville set, then, for any $\xi\in\sfS$, the set $\sfS\cup\{1/\xi\}$ is a Liouville set. 

\begin{proof}[Proof of Lemma $\ref{Lemma:unsuralpha}$]
Let $\uq=(q_n)_{n\ge 1}$ be an increasing sequence of positive integers such that, for sufficiently large $n$, 
$$
\Vert\ b_n\xi\Vert\le q_n^{-u_n},
$$
where $b_n \leq q_n^{\kappa_1}$. 
Write $\Vert b_n\xi\Vert=|b_n\xi-a_n|$ with $a_n\in\Z$. Since $\xi\not\in\Q$, the sequence $(|a_n|)_{n\ge 1}$ tends to infinity; in particular, for sufficiently large $n$, we have $a_n\not=0$. Writing
$$
\frac{1}{\xi}-\frac{b_n}{a_n}=\frac{-b_n}{\xi a_n}\left(\xi-\frac{a_n}{b_n}\right),
$$
one easily checks that, for sufficiently large $n$, 
$$
\Vert\ |a_n|\xi^{-1}\Vert\le |a_n|^{-u_n/2} 
\quad\hbox{and}\quad
1\le |a_n|<b_n^2 \leq q_n^{2\kappa_1}.
$$
\end{proof}

\begin{proof}[Proof of Theorem $\ref{Theorem:Field}$]
Let us check that for $\xi$ and $\xi'$ in $\Q\cup\sfS_{\uq, \uu}$, we have $\xi-\xi'\in \Q\cup\sfS_{\uq, \uu}$ and $\xi\xi'\in \Q\cup\sfS_{\uq,\uu}$. Clearly, it suffices to check
\\
$(1)$ For $\xi$ in $\sfS_{\uq, \uu}$ and $\xi'$ in $\Q$, we have $\xi-\xi'\in \sfS_{\uq, \uu}$ and $\xi\xi'\in \sfS_{\uq, \uu}$. 
\\
$(2)$ For $\xi$ in $\sfS_{\uq, \uu}$ and $\xi'$ in $\sfS_{\uq, \uu}$ with $\xi-\xi'\not\in\Q$, we have $\xi-\xi'\in \sfS_{\uq,\uu}$.
\\
$(3)$ For $\xi$ in $\sfS_{\uq, \uu}$ and $\xi'$ in $\sfS_{\uq, \uu}$ with $\xi\xi'\not\in\Q$, we have $\xi\xi'\in \sfS_{\uq, \uu}$. 

The idea of the proof is as follows. When $\xi\in \sfS_{\uq, \uu}$ is approximated by $a_n/b_n$ and when $\xi'=r/s\in \Q$, then $\xi-\xi'$ is approximated by $(sa_n-rb_n)/b_n$ and 
$\xi\xi'$ by $ra_n/sb_n$.
 When $\xi\in \sfS_{\uq, \uu}$ is approximated by $a_n/b_n$ and $\xi'\in \sfS_{\uq, \uu}$ by $a'_n/b'_n$, then 
 $\xi-\xi'$ is approximated by $(a_nb'_n-a'_n b_n)/b_nb'_n$ and 
$\xi\xi'$ by $a_na'_n/b_nb'_n$. The proofs which follow amount to writing down carefully these simple observations. 

 Let $\xi'' = \xi - \xi'$ and $\xi^* = \xi\xi'$. Then the sequence $(a_n'')$ and $(b_n'')$ are corresponding to $\xi''$; Similarly $(a_n^*)$ and $(b_n^*)$ corresponds to $\xi^*$.
 
Here is the proof of $(1)$. Let $\xi\in \sfS_{\uq, \uu}$ and $\xi'=r/s\in \Q$, with $r$ and $s$ in $\Z$, $s>0$. There are two constants $\kappa_1$ and $\kappa_2$ and there are sequences of rational integers 
 $\bigl(a_n\bigr)_{n\ge 1}$ and $\bigl(b_n\bigr)_{n\ge 1}$ such that 
 $$
 1\le b_n\le q_n^{\kappa_1} \quad \hbox{and}\quad 
 0<\bigl| b_n\xi - a_n\bigr|\le \frac{1}{q_n^{\kappa_2u_n}}\cdotp
 $$
 Let $\kappatilde_1>\kappa_1$ and $\kappatilde_2<\kappa_2$. Then, 
\begin{align}
\notag
 b_n''&= b_n^*= sb_n. 
\\
\notag 
a_n''&= sa_n-rb_n, 
\\
\notag
a_n^*&= ra_n.
\end{align}
Then one easily checks that, for sufficiently large $n$, we have
\begin{align}
\notag
 0<\bigl| b_n''\xi''- a_n''\bigr|
 =
s \, \bigl| b_n\xi - a_n\bigr|
 \le \frac{1}{q_n^{\kappa''_2u_n}},
 \\
 \notag
 0<\bigl| b_n^*\xi^* - a_n^*\bigr|
 =
|r| \, \bigl| b_n\xi - a_n\bigr|
 \le \frac{1}{q_n^{\kappa^*_2u_n}}\cdotp
\end{align}

Here is the proof of $(2)$ and $(3)$. Let $\xi$ and $\xi'$ be in $ \sfS_{\uq, \uu}$. There are constants 
$\kappa'_1$, $\kappa'_2$ $\kappa''_1$ and $\kappa''_2$
and there are sequences of rational integers 
 $\bigl(a_n\bigr)_{n\ge 1}$, $\bigl(b_n\bigr)_{n\ge 1}$,
 $\bigl(a_n'\bigr)_{n\ge 1}$ and $\bigl(b_n'\bigr)_{n\ge 1}$
 such that 
 $$
 1\le b_n\le q_n^{\kappa'_1} \quad \hbox{and}\quad 
 0<\bigl| b_n\xi - a_n\bigr|\le \frac{1}{q_n^{\kappa'_2u_n}},
 $$
 $$
 1\le b_n'\le q_n^{\kappa''_1} \quad \hbox{and}\quad 
 0<\bigl| b_n'\xi' - a_n'\bigr|\le \frac{1}{q_n^{\kappa''_2u_n}}\cdotp
 $$
Define 
$\kappatilde_1=\kappa'_1+\kappa''_1$ and let $\kappatilde_2>0$ satisfy $\kappatilde_2<\min\{\kappa'_2,\kappa''_2\}$. Set 
\begin{align}
\notag
 b_n''&= b_n^*= b_nb_n', 
\\
\notag 
a_n''&= a_nb_n'-b_n a_n', 
\\
\notag
a_n^*&= a_na_n'.
\end{align}
Then for sufficiently large $n$, we have
\begin{align}
\notag
 b_n''\xi'' - a_n''
 =
 b_n'\bigl( b_n \xi - a_n\bigr) - b_n \bigl( b_n' \xi' - a_n'\bigr)
\end{align}
and 
$$
 b_n^* \xi^* - a_n^* 
 =
b_n \xi \, \bigl( b_n' \xi' - a_n'\bigr)
+
a_n' \, \bigl( b_n \xi - a_n \bigr),
$$
hence
$$
\bigl|
b_n''\xi'' - a_n''\bigr|
 \le \frac{1}{q_n^{\kappatilde_2u_n}}
$$
and
 $$
\bigl| b_n^*\xi^*- a_n^* \bigr|
 \le \frac{1}{q_n^{\kappatilde_2u_n}}\cdotp
 $$
 Also we have 
 $$
 1\le b_n'' \le q_n^{\kappatilde_1}
 \quad
 \hbox{and}
 \quad
 1\le b_n^* \le q_n^{\kappatilde_1}.
 $$
 The assumption $\xi-\xi'\not\in\Q$ (resp $\xi\xi'\not\in\Q$) implies 
 $b_n'' \xi''\not= a_n''$ (respectively, $b_n^* \xi^*\not= a_n^*$). 
 Hence $\xi-\xi'$ and $\xi\xi'$ are in $\sfS_{\uq, \uu}$. This completes the proof of $(2)$ and $(3)$. 
 
 It follows from $(1)$, $(2)$ and $(3)$ that $\Q\cup\sfS_{\uq, \uu} $ is a ring.

Finally, if $\xi\in\Q\cup\sfS_{\uq, \uu}$ is not $0$, then $1/\xi\in\Q\cup\sfS_{\uq,\uu}$, by Lemma $\ref{Lemma:unsuralpha}$. This completes the proof of 
Theorem $\ref{Theorem:Field}$.
\end{proof}

\begin{remark}
Since the field $K_{\uq,\uu}$ does not contain irrational algebraic numbers, $2$ is not a square in $K_{\uq,\uu}$. For $\xi\in\sfS_{\uq,\uu}$, it follows that $\eta=2\xi^2$ is an element in $\sfS_{\uq, \uu}$ which is not the square of an element in $\sfS_{\uq,\uu}$. According to \cite{erd1}, we can write $\sqrt{2}=\xi_1\xi_2$ with two Liouville numbers $\xi_1,\xi_2$; then the set $\{ \xi_1,\xi_2 \} $ is not a Liouville set. 

\smallskip

Let $N$ be a positive integer such that $N$ cannot be written as a sum of two squares of an integer. 
Let us show that, for $\varrho \in\sfS_{\uq,\uu}$, the Liouville number $N \varrho^2\in\sfS_{\uq,\uu}$ is not the sum of two squares of elements in $\sfS_{\uq,\uu}$. Dividing by $\varrho^2$, we are reduced to show that the equation $N=\xi^2+(\xi')^2$ has no solution $(\xi,\xi')$ in $\sfS_{\uq,\uu}\times \sfS_{\uq,\uu}$. Otherwise, we would have, for suitable positive constants $\kappa_1$ and $\kappa_2$, 
\begin{align}
\notag
&
\left|
\xi-\frac{a_n}{b_n}\right|
\le \frac{1}{q_n^{\kappa_2u_n+1}},
\qquad 
1\le b_n \le q_n^{\kappa_1},
\\
\notag
&
\left|
\xi'-\frac{a_n'}{b_n'}\right|
\le \frac{1}{q_n^{\kappa_2u_n+1}},
\qquad 
1\le b_n'\le q_n^{\kappa_1},
\end{align}
hence
$$
\left|
\xi^2 -\frac{
a_n^2}
{ b_n ^2}
\right|
\le
\frac{2|\xi|+1}{q_n^{ \kappa_2u_n+1}},
\qquad 
\left|
(\xi')^2 -\frac{
(a_n')^2}
{ (b_n')^2}
\right|
\le
\frac{2|\xi'|+1}{q_n^{\kappa_2u_n+1}}
$$
and
$$
\left|
\xi^2+(\xi')^2-\frac{
\bigl(
a_nb_n'\bigr)^2+\bigl(a_n'b_n\bigr)^2}
{\bigl(b_nb_n'\bigr)^2}
\right|
\le
\frac{2(|\xi|+|\xi'|+1)}{q_n^{\kappa_2u_n+1}}\cdotp
$$
Using $\xi^2+(\xi')^2=N$, we deduce 
$$
\bigl|
N{\bigl(b_nb_n'\bigr)^2-\bigl(
a_nb_n'\bigr)^2-\bigl(a_n'b_n\bigr)^2}\bigr|<1.
$$
The left hand side is an integer, hence it is $0$:
$$
N{\bigl(b_nb_n'\bigr)^2=
\bigl(
a_nb_n'\bigr)^2+\bigl(a_n'b_n\bigr)^2}.
$$
This is impossible, since the equation $x^2+y^2=Nz^2$ has no solution in positive rational integers. 

Therefore, if we write $N=\xi^2+(\xi')^2$ with two Liouville numbers $\xi,\xi'$, which is possible by the above mentioned result from P.~Erd\H{o}s \cite{erd1},
then the set $\{ \xi,\xi' \} $ is not a Liouville set. 
\end{remark}
 
\section{Proof of Theorem $\ref{Theorem::UncountableLiouvilleSets}$}
\label{Section:UncountableLiouvilleSetss}

We first prove the following lemma which will be required for the proof of part $(ii)$ of Theorem $\ref{Theorem::UncountableLiouvilleSets}$.

\begin{lemma}\label{lemma:ifSqnotempty}
Let $\xi$ be a real number, $n$, $q$ and $q'$ be positive integers. Assume that there exist rational integers $p$ and $p'$ such that $p /q\not=p'/q'$ and 
$$ 
|q \xi -p | \le \frac{1}{q^{u_n}}, \quad |q' \xi -p'| \le \frac{1}{(q')^{ {u_n}+1}}\cdotp
$$
Then we have 
$$
\hbox{either} \quad
q'\ge q^{{u_n}}\quad \hbox{or} \quad q \ge (q')^{u_n}.
$$
\end{lemma}

\begin{proof}[Proof of Lemma $\ref{lemma:ifSqnotempty}$]
From the assumptions we deduce 
$$
\frac{1}{q q'}\le
\frac{|p q'-p'q |}{q q'} 
\le 
\left|\xi-\frac{p }{q }\right|+
\left|\xi-\frac{p'}{q'}\right|
\le \frac{1}{q ^{{u_n}+1}}+\frac{1}{(q')^{{u_n}+2}},
$$
hence
$$
q^{u_n}(q')^{{u_n}+1} \le (q')^{{u_n}+2}+q^{{u_n}+1}.
$$
If $q< q' $, we deduce 
$$
q^{u_n}\le q'+
\left(\frac{q}{q'}\right)^{u_n+1} <q'+1.
$$
Assume now $q\ge q' $. Since the conclusion of Lemma $\ref{lemma:ifSqnotempty}$ is trivial if $u_n=1$ and also if $q'=1$, we assume $u_n > 1$ and $q'\ge 2$. 
From 
$$
q^{u_n}(q')^{u_n+1} \le (q')^{u_n+2}+q^{u_n+1}\le (q')^2q^{u_n}+q^{u_n+1}
$$
we deduce 
$$
 (q')^{u_n+1} - (q')^2 \le q.
$$
From $(q')^{u_n-1}>(q')^{u_n-2}$ we deduce $(q')^{u_n-1}\ge (q')^{u_n-2}+1$, which we write as
$$
(q')^{u_n+1}-(q')^2\ge (q')^{u_n}.
$$
Finally
$$
(q')^{u_n} \le 
(q')^ {u_n+1}- (q')^{2}
\le q.
$$
\end{proof}

\begin{proof}[Proof of Theorem $\ref{Theorem::UncountableLiouvilleSets}$]
$\phantom{.}$ 
 Suppose $\displaystyle\limsup_{n\to\infty}\frac{\log q_{n+1}}{u_n\log q_n} = 0$. Then, 
we get, $$
\lim_{n\rightarrow\infty} \frac{\log q_{n+1}}{u_n\log q_n}=0.
$$
Suppose $\sfS_{\uq, \uu} \ne \emptyset.$ Let $\xi \in \sf {S}_{\uq, \uu}$. 
From Remark $\ref{Remark:definitionLiouvilleSet}$, it follows that there exists a sequence $\bigl(b_n\bigr)_{n\ge 1}$ of positive integers and there exist two positive constants $\kappa_1$ and $\kappa_2$ such that, for any sufficiently large $n$, 
$$
q_{n} \leq b_{n} \leq q_{n}^{\kappa_1} \mbox{ and } \Vert b_n \xi \Vert \le q_n^{- \kappa_2 u_n}\cdotp
$$
Let $n_0$ be an integer $\ge \kappa_1$ such that these inequalities are valid for $n\ge n_0$ and such that, for $n\ge n_0$, $q_{n+1}^{\kappa_1}<q_n^{u_n}$ (by the assumption). 
Since the sequence $(q_n)_{n\ge 1}$ is increasing, we have $q_n^{\kappa_1}<q_{n+1}^{u_n}$ for $n\ge n_0$. From the choice of $n_0$ we deduce 
$$
b_{n+1}\le q_{n+1}^{\kappa_1}<q_n^{u_n }\le b_n^{u_n}
$$
and
$$
b_n\le q_n^{\kappa_1}<q_{n+1}^{u_n} \le b_{n+1}^{u_n}
$$
for any $n\ge n_0$. Denote by $a_n$ (resp.~$a_{n+1}$) the nearest integer to $\xi b_n$ (resp.~ to $\xi b_{n+1}$).
Lemma $\ref{lemma:ifSqnotempty}$ with $q$ replaced by $b_n$ and $q'$ by $b_{n+1}$ implies that for each $n\ge n_0$, 
$$
\frac{a_n}{b_n}=
\frac{a_{n+1}}{b_{n+1}}\cdotp
$$
This contradicts the assumption that $\xi$ is irrational. This proves that $\sfS_{\uq, \uu} = \emptyset.$

\bigskip

\noindent Conversely, assume 
$$
\limsup_{n\rightarrow\infty} \frac{\log q_{n+1}}{u_n\log q_n}>0.
$$
Then there exists $\vartheta>0$ and there exists a sequence $(N_\ell)_{\ell\ge 1}$ of positive integers such that 
$$
q_{N_\ell}>q_{N_\ell-1}^{\vartheta (u_{N_\ell-1)}}
$$
for all $\ell\ge 1$. Define a sequence $(c_\ell)_{\ell\ge 1}$ of positive integers by
$$
2^{c_\ell}\le q_{N_\ell} < 2^{c_\ell+1}.
$$
Let $\ue=(e_\ell)_{\ell\ge 1}$ be a sequence of elements in $\{-1,1\}$. Define 
$$
\xi_{\ue}=\sum_{\ell\ge 1} \frac{e_\ell}{2^{c_\ell}}\cdotp
$$
It remains to check that $\xi_{\ue}\in \sfS_{\uq,\uu}$ and that distinct $\ue$ produce distinct $\xi_{\ue}$. 

Let $\kappa_1=1$ and let $\kappa_2$ be in the interval $0<\kappa_2<\vartheta$. For sufficiently large $n$, let $\ell$ be the integer such that $N_{\ell-1}\le n<N_\ell$. Set
$$
b_n=2^{c_{\ell-1}}, \quad a_n=\sum_{h=1}^{\ell-1} e_h 2^{c_{\ell-1}-c_h},
\quad r_n=\frac{a_n}{b_n}\cdotp
$$
We have 
$$
\frac{1}{2^{c_\ell}}<
\left|
\xi_{\ue}-r_n\right|
=
\left|
\xi_{\ue}- 
\sum_{h\ge \ell} \frac{e_h}{2^{c_h}} \right|
\le 
\frac{2}{2^{c_\ell}}\cdotp
$$
Since $\kappa_2<\vartheta$, $n$ is sufficiently large and 
$n\le N_\ell-1$,
we have 
$$
4q_n^{\kappa_2 u_n} \le 
4q_{N_\ell-1}^{\kappa_2u_{N_\ell-1}}\le 
q_{N_\ell},
$$
hence
$$
\frac{2}{2^{c_\ell}}<\frac{4}{q_{N_\ell}}<\frac{1}{q_n^{\kappa_2 u_n}}
$$
for sufficiently large $n$. This proves 
$\xi_{\ue}\in \sfS_{\uq, \uu}$ and hence $ \sfS_{\uq, \uu}$ is not empty. 

\bigskip

Finally, if $\ue$ and $\ue'$ are two elements of $\{-1,+1\}^{\N}$ for which $e_h=e'_h$ for $1\le h< \ell$ and, say, $e_\ell=-1$, $e'_\ell=1$, then 
$$
\xi_{\ue}< \sum_{h=1}^{\ell-1} \frac{e_h}{2^{c_h}} <\xi_{\ue'},
$$
hence $\xi_{\ue}\not= \xi_{\ue'}$. This 
completes the proof of Theorem $\ref{Theorem::UncountableLiouvilleSets}$.

\end{proof}

\section{Proof of Corollary $\ref{Corollary:EquivalenceTrivial}$} \label{Section:EquivalenceTrivial}

The proof of Corollary $\ref{Corollary:EquivalenceTrivial}$ as a consequence of Theorem $\ref{Theorem::UncountableLiouvilleSets}$ relies on the following elementary lemma. 
\begin{lemma}\label{Lemma:equivalencetrivial}
Let $(a_n)_{n\ge 1}$ and $(b_n)_{n\ge 1}$ be two increasing
sequences of positive integers. Then there exists an increasing sequence of positive integers $(q_n)_{n\ge 1}$ satisfying the following properties:
\\
$(i)$ The sequence $(q_{2n})_{n\ge 1}$ is a subsequence of the sequence $(a_n)_{n\ge 1}$.
\\
$(ii)$ The sequence $(q_{2n+1})_{n\ge 0}$ is a subsequence of the sequence $(b_n)_{n\ge 1}$.
\\
$(iii)$ For $n\ge 1$, $q_{n+1}\ge q_n^n$. 

\end{lemma}

\begin{proof}[Proof of Lemma $\ref{Lemma:equivalencetrivial}$]
We construct the sequence $(q_n)_{n\ge 1}$ inductively, starting with $q_1=b_1$ and with $q_2$ the least integer $a_i$ satisfying $a_i\ge b_1$. Once $q_n$ is known for some $n\ge 2$, we take for $q_{n+1}$ the least integer satisfying the following properties: 
\\ 
$\bullet$
$q_{n+1}\in \{a_1,a_2,\dots\}$ if $n$ is odd, 
$q_{n+1}\in \{b_1,b_2,\dots\}$ if $n$ is even.
\\
$\bullet$ $q_{n+1}\ge q_n^n$.
\end{proof}

\begin{proof}[Proof of Corollary $\ref{Corollary:EquivalenceTrivial}$]
Let $\xi$ and $\eta$ be Liouville numbers. There exist two sequences of positive integers $(a_n)_{n\ge 1}$ and $(b_n)_{n\ge 1}$, which we may suppose to be increasing, such that 
$$
\Vert a_n \xi \Vert \le a_n^{-n}
\quad\hbox{and}\quad
\Vert b_n \eta \Vert \le b_n^{-n}
$$
for sufficiently large $n$.
Let $\uq=(q_n)_{n\ge 1}$ be an increasing sequence of positive integers satisfying the conclusion of Lemma $\ref{Lemma:equivalencetrivial}$. According to 
Theorem $\ref{Theorem::UncountableLiouvilleSets}$, the Liouville set $\sfS_{\uq}$ is not empty. Let $\varrho\in \sfS_{\uq}$. Denote by $\uq'$ the subsequence $(q_2,q_4,\dots,q_{2n},\dots)$ of $\uq$ and by $\uq''$ the subsequence $(q_1,q_3,\dots,q_{2n+1},\dots)$. We have $\varrho\in \sfS_{\uq} = \sfS_{\uq'}\cap \sfS_{\uq''}$. Since the sequence $(a_n)_{n\ge 1}$ is increasing, we have $q_{2n}\ge a_n$, hence $\xi\in \sfS_{\uq'}$. Also, since the sequence $(b_n)_{n\ge 1}$ is increasing, we have $q_{2n+1}\ge b_n$, hence $\eta\in \sfS_{\uq''}$.
Finally, $\xi$ and $\varrho $ belong to the Liouville set $\sfS_{\uq'}$, while 
$\eta$ and $\varrho $ belong to the Liouville set $\sfS_{\uq''}$.
\end{proof}

\section{Proofs of Propositions $\ref{proposition:UncountablymanyLiouvilleSets}$, $\ref{proposition:SubLiouvilleSets}$, $\ref{proposition:13}$ and $\ref{proposition:14}$ }
\label{section:PfPropositions}

\begin{proof}[Proof of Proposition $\ref{proposition:UncountablymanyLiouvilleSets}$]
The fact that for $0<\tau<1$ the set $\sfS_{\uq^{(\tau)}}$ is not empty 
follows from Theorem $\ref{Theorem::UncountableLiouvilleSets}$, since
$$
\lim_{n\rightarrow\infty} 
\frac{
\log q^{(\tau)}_{n+1}}
{
n\log q^{(\tau)}_n}
=1.
$$
In fact, if $(e_n)_{n\ge 1}$ is a bounded sequence of integers with infinitely many nonzero terms, then
$$
\sum_{n\ge 1} \frac{e_n}{q^{(\tau)}_n}\in \sfS_{\uq^{(\tau)}}.
$$

Let $0<\tau_1<\tau_2<1$.
For $n\ge 1$, define 
$$
q_{2n}=q_n^{(\tau_1)}=2^{n! \lfloor n^{\tau_1} \rfloor}
\quad\hbox{and}\quad
q_{2n+1}=q_n^{(\tau_2)}=2^{n! \lfloor n^{\tau_2} \rfloor}.
$$
One easily checks that $(q_m)_{m\ge 1}$ is an increasing sequence with 
$$
\frac{\log q_{2n+1}}{n\log q_{2n}}\rightarrow 0
\quad\hbox{and}\quad
\frac{\log q_{2n+2}}{n\log q_{2n+1}}\rightarrow 0.
$$
From Theorem $\ref{Theorem::UncountableLiouvilleSets}$ one deduces $\sfS_{\uq^{(\tau_1)}}\cap \sfS_{\uq^{(\tau_2)}}=\emptyset$.

\end{proof}

\begin{proof}[Proof of Proposition $\ref{proposition:SubLiouvilleSets}$]
For sufficiently large $n$, define 
$$
a_n=\sum_{m= 1}^n 2^{(2n)! -(2m-1)!\lambda_m}.
$$ 
Then
 $$
 \frac{1}{q_{2n}^{(2n+1)\lambda_{n+1}} } < \xi-\frac{a_n}{q_{2n}}
 =\sum_{m\ge n+1} \frac{1}{2^{(2m-1)!\lambda_m}}
 \le \frac{2}{q_{2n}^{(2n+1)\lambda_{n+1}} } \cdotp
 $$
The right inequality with the lower bound $\lambda_{n+1}\ge 1$ proves that $\xi\in \sfS_{\uq'}$. 

Let $\kappa_1$ and $\kappa_2$ be positive numbers, $n$ a sufficiently large integer, $s$ an integer in the interval $q_{2n+1}\le s\le q_{2n+1}^{\kappa_1}$ and $r$ an integer. Since $\lambda_{n+1}<\kappa_2n$ for sufficiently large $n$, we have 
$$
q_{2n}^{(2n+1)\lambda_{n+1}}<q_{2n}^{\kappa_2 n (2n+1)}= q_{2n+1}^{\kappa_2 n } \le s^{\kappa_2 n }.
$$
Therefore, if $r/s=a_n/q_{2n}$, then 
$$
\left|\xi-\frac{r}{s}\right| = 
\left|\xi-\frac{a_n}{q_{2n}}\right| > \frac{1}{q_{2n}^{(2n+1)\lambda_{n+1}} } >\frac{1}{s^{\kappa_2n}}\cdotp
$$ 
On the other hand, for $r/s\not =a_n/q_{2n}$, we have 
$$
\left|\xi-\frac{r}{s}\right| \ge 
\left|\frac{a_n}{q_{2n}}-\frac{r}{s}\right|-
\left| \xi-\frac{a_n}{q_{2n}}\right| \ge \frac{1}{q_{2n} s} - \frac{2}{q_{2n}^{(2n+1)\lambda_{n+1}} }\cdotp
$$ 
Since $\lambda_n\rightarrow\infty$, for sufficiently large $n$ we have 
$$
4q_{2n} s\le 4q_{2n} q_{2n+1}^{\kappa_1} 
= 4 q_{2n}^{1+\kappa_1(2n+1)}
\le q_{2n}^{(2n+1)\lambda_{n+1}}
$$
hence
$$
\frac{2}{q_{2n}^{(2n+1)\lambda_{n+1}} }\le \frac{1}{2q_{2n} s} \cdotp
$$
Further
$$
2q_{2n} <q_{2n+1}<q_{2n+1}^{\kappa_2n-1}\le s^{\kappa_2n-1}.
$$
Therefore 
$$
\left|\xi-\frac{r}{s}\right| \ge 
 \frac{1}{2q_{2n} s} >\frac{1}{s^{\kappa_2n}},
 $$
 which shows that $\xi\not\in \sfS_{\uq''}$.
\end{proof}

\begin{proof}[Proof of Proposition $\ref{proposition:13}$]
Let $(\lambda_s)_{s\ge 0}$ be a strictly increasing sequence of positive rational integers with $\lambda_0=1$. Define two sequences $(n'_k)_{k\ge 1}$ and $(n''_h)_{h\ge 1}$ of positive integers as follows. The sequence $(n'_k)_{k\ge 1}$ is the increasing sequence of the positive integers $n$ for which there exists $s\ge 0$ with 
$\lambda_{2s}\le n<\lambda_{2s+1}$, while $(n''_h)_{h\ge 1}$ is the increasing sequence of the positive integers $n$ for which there exists $s\ge 0$ with 
$\lambda_{2s+1}\le n<\lambda_{2s+2}$.

For $s\ge 0$ and $\lambda_{2s}\le n<\lambda_{2s+1}$, set 
$$
k=n-\lambda_{2s}+\lambda_{2s-1}-\lambda_{2s-2}+\cdots+\lambda_1.
$$
Then $n=n'_k$. 

For $s\ge 0$ and $\lambda_{2s+1}\le n<\lambda_{2s+2}$, set 
$$
h=n-\lambda_{2s+1}+\lambda_{2s}-\lambda_{2s-1}+\cdots-\lambda_1+1.
$$
Then $n=n''_h$. 

For instance, when $\lambda_s=s+1$, the sequence $(n'_k)_{k\ge 1}$ is the sequence $(1,3,5\dots)$ of odd positive integers, while $(n''_h)_{h\ge 1}$ is the sequence $(2,4,6\dots)$ of even positive integers. Another example is $\lambda_s=s!$, which occurs in the paper \cite{erd1} by Erd\H{o}s. 

In general, for $n=\lambda_{2s}$, we write $n=n'_{k(s)}$ where 
$$
k(s)=\lambda_{2s-1}-\lambda_{2s-2}+\cdots+\lambda_1<\lambda_{2s-1}.
$$
Notice that $\lambda_{2s}-1=n''_h$ with $h=\lambda_{2s}-k(s)$. 

Next, define two increasing sequences $(d_n)_{n\ge 1}$ and $\uq=(q_n)_{n\ge 1}$ of positive integers by induction, with $d_1=2$,
$$
d_{n+1}=
\begin{cases}
k d_n &\hbox{if $n=n'_k$},\\
h d_n &\hbox{if $n=n''_h$}
\\
\end{cases}
$$
for $n\ge 1$ and $q_n=2^{d_n}$. Finally, let $\uq'=(q'_k)_{k\ge 1}$ and $\uq''=(q''_h)_{h\ge 1}$ be the two subsequences of $\uq$ defined by 
$$
q'_k=q_{n'_k},\quad k\ge 1, \qquad
q''_h=q_{n''_h},\quad h\ge 1.
$$
Hence $\uq$ is the union of theses two subsequences. Now we check that
the number
$$
\xi=\sum_{n\ge 1} \frac{1}{q_n}
$$
belongs to $\sfS_{\uq'}\bigcap \sfS_{\uq''}$. Note that by Theorem $\ref {Theorem::UncountableLiouvilleSets}$ that $\sfS_{\uq} \ne \emptyset$ as $\sfS_{\uq'}\ne \emptyset$ and $\sfS_{\uq''}\ne \emptyset$. 
Define 
$$
a_n=\sum_{m=1}^n 2^{d_n-d_m}.
$$
Then
$$
\frac{1}{q_{n+1}}<\xi-\frac{a_n}{q_n} = \sum_{m\ge n+1} \frac{1}{q_m}< \frac{2}{q_{n+1}}\cdotp
$$
If $n=n'_k$, then 
$$
\left|
\xi-\frac{a_{n'_k}}{q'_k}
\right|
<\frac{2}{(q'_k)^k}
$$
while if $n=n''_h$, then 
$$
\left|
\xi-\frac{a_{n''_h}}{q''_h}
\right|
<\frac{2}{(q''_h)^h}\cdotp
$$
This proves $\xi\in\sfS_{\uq'}\cap \sfS_{\uq''}$.

Now, we choose $\lambda_s=2^{2^s}$ for $s\ge 2$ and we prove that $\xi$ does not belong to $\sfS_{\uq}$. 
Notice that $\lambda_{2s-1}=\sqrt{\lambda_{2s}}$.
Let $n=\lambda_{2s}=n'_{k(s)}$. We have 
$k(s)<\sqrt{\lambda_{2s}}$ and 
$$
\left|
\xi-\frac{a_{n}}{q_n}
\right|
>\frac{1}{q_{n+1}}=\frac{1}{q_{n}^{k(s)}}>\frac{1}{q_{n}^{\sqrt{n}}}\cdotp
$$
Let $\kappa_1$ and $\kappa_2$ be two positive real numbers and assume $s$ is sufficiently large. Further, let $u/v\in\Q$ with $v\le q_n^{\kappa_1}$. If $u/v=a_n/q_n$, then
$$
\left|
\xi-\frac{u}{v}
\right|=
\left|
\xi-\frac{a_{n}}{q_n}
\right|
> \frac{1}{q_{n}^{\sqrt{n}}}> \frac{1}{q_{n}^{\kappa_2n}}\cdotp
$$
On the other hand, if $u/v\not=a_n/q_n$, then
$$
\left|
\xi-\frac{u}{v}
\right|
\ge 
\left|
 \frac{u}{v}- \frac{a_{n}}{q_n}
 \right|
 -
\left| 
\xi-\frac{a_{n}}{q_n}
\right|
$$
with
$$
\left|
 \frac{u}{v}- \frac{a_{n}}{q_n}
 \right|\ge \frac{1}{vq_n}\ge \frac{1}{q_n^{\kappa_1+1}}
 >\frac{2}{q_n^{\sqrt{ n}}} 
 $$
 and 
 $$
\left| 
\xi-\frac{a_{n}}{q_n}
\right|
<\frac{1}{q_{n}^{\sqrt{n}}}\cdotp
 $$
 Hence
 $$
 \left|\xi-\frac{u}{v}\right|>
 \frac{1}{q_{n}^{\sqrt{n}}}
 >\frac{1}{q_n^{\kappa_2 n}}\cdotp
 $$
 This proves Proposition $\ref{proposition:13}$.
\end{proof}

 \begin{proof}[Proof of Proposition $\ref{proposition:14}$]
 
 Let $\uu =(u_n)_{n\geq 1}$ be a sequence of positive real numbers such that $\sqrt{u_{n+1}} \leq u_n+1\leq u_{n+1}$. We prove more precisely that for any sequence $\uq$ such that $q_{n+1}>q_n^{u_n}$ for all $n\ge 1$, the sequence $\uq'=(q_{2m+1})_{m\ge 1}$ has 
$\sfS_{\uq', \uu}\not=\sfS_{\uq,\uu}$. This implies the proposition, since any increasing sequence has a subsequence satisfying $q_{n+1}>q_n^{u_n}$.

Assuming $q_{n+1}>q_n^{u_n}$ for all $n\ge 1$, we define
$$
d_n=\begin{cases}
q_n & \hbox{for even $n$,}\\
q_{n-1}^{\lfloor \sqrt{u_n}\rfloor} & \hbox{for odd $n$.}
\end{cases}
$$
We check that the number 
$$
\xi=\sum_{n\ge 1} \frac{1}{d_n}
$$
satisfies $\xi\in \sfS_{\uq', \uu}$ and $\xi\not\in \sfS_{\uq, \uu}$. 

Set $b_n=d_1d_2\cdots d_n$ and 
$$
a_n=
\sum_{m=1}^{n} \frac{b_n}{d_m}
=
\sum_{m=1}^{n} \prod_{
\atop{1\le i\le n}, {i\not=m}} d_i,
$$
so that
$$
\xi-\frac{a_n}{b_n}=\sum_{m\ge n+1} \frac{1}{d_m}
\cdotp
$$
It is easy to check from the definition of $d_n$ and $q_n$ that we have, for sufficiently large $n$, 
$$
b_n\le q_1\cdots q_n\le q_{n-1}^{u_{n-1}}q_n\le q_n^2
$$
and
$$
\frac{1}{d_{n+1}}\le \xi-\frac{a_n}{b_n}\le \frac{2}{d_{n+1}}\cdotp
$$
For odd $n$, since $d_{n+1}=q_{n+1}\ge q_n^{u_n}$, we deduce
$$
\left|\xi-\frac{a_n}{b_n} \right| \le \frac{2}{q_n^{u_n}},
$$
hence $\xi\in \sfS_{\uq',\uu}$. 

\bigskip

For even $n$, we plainly have
$$
\left|\xi-\frac{a_n}{b_n} \right| > \frac{1}{d_{n+1}}=
\frac{1}{q_{n}^{\lfloor \sqrt{u_{n+1}}\rfloor} } \cdotp
$$ 
Let $\kappa_1$ and $\kappa_2$ be two positive real numbers, and let $n$ be sufficiently large. 
Let $s$ be a positive integer with $s\le q_n^{\kappa_1}$ and let $r$ be an integer. If $r/s= a_n/b_n$, then
$$
\left|\xi-\frac{r}{s} \right| =
\left|\xi-\frac{a_n}{b_n} \right| > \frac{1}{q_{n}^{\kappa_2 u_n}}\cdotp
$$
Assume now
$r/s\not= a_n/b_n$. From
$$
\left|\xi-\frac{a_n}{b_n} \right| \le
\frac{2}{q_{n}^{\lfloor \sqrt{u_{n+1}}\rfloor} }
\le 
 \frac{1}{2q_n^{\kappa_1+2}},
$$
we deduce
$$
\frac{1}{q_n^{\kappa_1+2}}\le \frac{1}{sb_n}\le 
\left|\frac{r}{s}-\frac{a_n}{b_n}\right|\le
\left|\xi-\frac{r}{s} \right| +
\left|\xi-\frac{a_n}{b_n} \right| 
\le \left|\xi-\frac{r}{s} \right| + \frac{1}{2q_n^{\kappa_1+2}},
$$
hence 
$$
\left|\xi-\frac{r}{s} \right| \ge \frac{1}{2q_n^{\kappa_1+2}}> \frac{1}{q_{n}^{\kappa_2 u_n}}\cdotp
$$
This completes the proof that $\xi\not\in \sfS_{\uq, \uu}$. 
 \end{proof}

\section{Proof of Proposition $\ref{proposition:Squ-densebutnotGdelta}$} \label{section:Squ-densebutnotGdelta}
\begin{proof}[Proof of Proposition $\ref{proposition:Squ-densebutnotGdelta}$]

If $\sfS_{\uq,\uu}$ is non empty, let
 $\gamma \in \sfS_{\uq, \uu}$. By Theorem \ref{Theorem:Field}, $\gamma+\Q$ is contained in $\sfS_{\uq,\uu}$, hence $\sfS_{\uq,\uu}$ is dense in $\R$.

Let $t$ be an irrational real number which is not Liouville. Hence $t\not\in \sfK_{\uq,\uu}$, and therefore, 
by Theorem \ref{Theorem:Field}, $\sfS_{\uq,\uu}\cap (t+\sfS_{\uq,\uu})=\emptyset$. This implies that $\sfS_{\uq,\uu}$ is not a $G_\delta$ dense subset of $\R$.

\end{proof}

\vfill

\end{document}